\newif\ifdraft

\ifdraft

\documentclass[11pt, reqno]{amsart}
\usepackage{lmodern}

\usepackage[a4paper, margin=1in]{geometry}

\else

\documentclass[reqno]{amsart}
\usepackage{lmodern}

\usepackage[a4paper, margin=.75in]{geometry}
\fi

\usepackage{amsmath, amsthm, thmtools, amsfonts, amssymb, mathtools}
\usepackage{pdflscape, blkarray, multirow, booktabs}
\usepackage{mathrsfs}
\usepackage{cancel}

\usepackage{amstext} 
\usepackage{array}   
\newcolumntype{L}{>{$}l<{$}} 

\usepackage[dvipsnames]{xcolor}
\usepackage{hyperref}
\hypersetup{
	colorlinks = true,
	linkcolor = {Blue},
	citecolor={BrickRed},
}
\usepackage{cleveref}

\usepackage{enumerate}
\usepackage{makecell}
\usepackage{graphicx}
\usepackage[all]{xy}
\usepackage{float}
\usepackage{tikz}
\usetikzlibrary{decorations.pathreplacing, arrows.meta}
\usepackage{spectralsequences}
\usepackage{tikz-cd}

\numberwithin{equation}{section}

\theoremstyle{plain}

\newcommand{\R}{\mathbb{R}}
\newcommand{\Z}{\mathbb{Z}}

\theoremstyle{definition}

\newtheorem{theorem}{Theorem}[section]

\newtheorem{lemma}[theorem]{Lemma}
\newtheorem{remark}[theorem]{Remark}
\newtheorem{eg}[theorem]{Example}

\newtheorem{thmIntro}{Theorem}

\newtheorem*{theorem*}{Theorem}

\newcommand{\nocontentsline}[3]{}
\newcommand{\tocless}[2]{\bgroup\let\addcontentsline=\nocontentsline#1{#2}\egroup}
\definecolor{newBlue}{HTML}{8611e8}

\newcommand{\hf}{\hspace{0.2cm}}

\usepackage{import}
\usepackage{xifthen}
\usepackage{pdfpages}
\usepackage{transparent}

\begin{document}
\allowdisplaybreaks
\title{On Manifolds homeomorphic to spheres}

\author[S. Basu]{Somnath Basu}
\address{Department of Mathematics and Statistics, IISER Kolkata,\\ West Bengal, India}
\email{somnath.basu@iiserkol.ac.in}

\author[S. Prasad]{Sachchidanand Prasad}
\address{Department of Mathematics, IIIT Delhi\\ New Delhi, India}
\email{sachchidanand@iiitd.ac.in}



\begin{abstract}
    We prove a result analogous to Reeb's theorem in the context of Morse-Bott functions: if a closed, smooth manifold $M$ admits a Morse-Bott function having two critical submanifolds $S^k$ and $S^l$ ($k \neq l$), then $M$ has dimension $k+l+1$ and is homeomorphic to the standard sphere $S^{k+l+1}$ but not necessarily diffeomorphic to it. We also prove similar results for projective spaces over the real numbers, complex numbers and quaternions. 
\end{abstract}

\date{\today}
\maketitle

\frenchspacing 

\section{Introduction}\label{sec:introduction}
\hspace{0.3cm} Morse theory is a powerful tool in differential topology, historically used by Marston Morse to show how the critical points of a smooth function can describe the topology of a manifold. One of its important results is Reeb's theorem, proved by Georges Reeb \cite{Reeb46} in 1946, which says that a compact, smooth manifold with a Morse function having two critical points must be homeomorphic to a sphere~\cite[Theorem 4.1]{Mil63}. This result had been crucially utilized by Milnor in his landmark work on the discovery of exotic $7$-spheres~\cite{Mil56}. A study of Morse-Bott functions with two critical values on a closed surfaces has been studied in the context of Reeb's graph \cite[Proposition 4.1]{Gel21}. Eells-Kuiper \cite{EeKui62} analyzed Morse functions with three critical points and showed that such a connected manifold has the cohomology ring of a projective plane (real, complex, quaternion or octonion). 

\hspace{0.3cm} The critical points of a Morse function are isolated and non-degenerate. The condition of non-degeneracy has to be modified in geometric scenarios where functions have symmetries and critical sets are submanifolds (of varying dimensions). This sets the stage for the theory of Morse-Bott functions which was developed by Raoul Bott \cite{Bot59}. Let $M$ be a Riemannian manifold and $f:M\to\mathbb{R}$ be a smooth function. Let $\mathrm{Cr}(f) $ denotes the set of critical points of $f$. Let $N$ be any connected submanifold of $M$ such that $N\subseteq \mathrm{Cr} (f)$. If $\nu$ denotes the normal bundle of $N$, then for any point $p\in N$, we have a decomposition $T_p M = T_p N \oplus \nu _p$. Note that for $p \in N$, and $V \in T_p N,  W\in T_p M$, the Hessian vanishes, that is, $\mathrm{Hess}_p(f)(V,W) = 0 $. Therefore, $\mathrm{Hess} _p(f)$ is characterized by its restriction to $\nu _p$. The submanifold $N$ is said to be a non-degenerate critical submanifold of $f$ if for any $p \in N$, the Hessian restricted to $\nu_p$ is non-degenerate. The function $f: M\to \mathbb{R} $ is said to be Morse-Bott if $\mathrm{Cr} (f)$ is the disjoint union of connected non-degenerate submanifolds. Morse functions are the first examples of Morse-Bott functions. If $f : M \to \mathbb{R} $ is a Morse function and $\pi : E \to  M$ is any smooth fibre bundle, then the composition $\pi {\scriptstyle\circ} f: E \to  \mathbb{R}$ is a Morse-Bott function whose critical submanifolds are exactly the fibre over critical points of $f$.          

\hspace{0.3cm} Given Reeb's theorem, it is natural to ask: can we characterize the topology of a smooth manifold if it admits a Morse-Bott function with only two (connected) critical submanifolds? This question, though broad, invites us to explore how the topology of critical submanifolds impose constraints on the topology of the manifold. Formulating precise analogues of Reeb's theorem for Morse-Bott functions with two (or more) critical submanifolds remains an open terrain. In this article, we prove the following.

\begin{thmIntro}[\autoref{thm:forSpheres}, \autoref{eg:ht-square-v2}]\label{thm:introSphere}
    Let $M$ be a closed, smooth manifold of dimension $d$. Let $f$ be a Morse-Bott function on $M$ with only critical submanifolds $S^k=f^{-1}(-1)$ and $S^l=f^{-1}(1)$ with $k\neq l$. Then $d = k+l+1$ and $M$ is homeomorphic to $S^d$.
\end{thmIntro}
\noindent The above result includes the case of $k=0$. \autoref{thm:introSphere} is sharp in the sense that the homeomorphism cannot, in general, be extended to a diffeomorphism. In forthcoming joint work of the second-named author with A. Bhowmick and T. Schick, it is shown that for a compact manifold $M$, given two disjoint, closed, connected, embedded submanifolds $N_1$ and $N_2$, there exists a Riemannian metric $g$ on $M$ such that the cut locus of $N_1$ is exactly $N_2$ and vice versa, if and only if there exists a Morse-Bott function $f : M \to \mathbb{R}$ with exactly two critical submanifolds, $N_1$ and $N_2$. Furthermore, they show that on every exotic sphere $\Sigma^d$, there exists a Riemannian metric $g$ and disjoint smooth embeddings $S^k, S^l \hookrightarrow \Sigma^d$ with $d = k + l + 1$ such that the embedded $S^k$ is precisely the cut locus of $S^l$. We have also indicated a more explicit example (see \autoref{eg:sigma7-as-join}).

\hspace{0.3cm} \autoref{thm:introSphere} need not be true if we assume $ k =l $; see the example of the height function on a torus lying horizontally (\autoref{fig:torusOnHorizontalPlane}). For $k = l$, in a more general sense, Lerario, Meroni and Zuddas \cite{LeMeZu24} have shown the following:
\begin{theorem*}\cite[Corollary 16]{LeMeZu24}
    Let $M$ be a smooth closed connected orientable manifold of dimension $n\geq 6$, and let $1 \leq k < n$. Further assume that $f: M \to  \mathbb{R} $ be a smooth function with two critical submanifolds, each being $S^k$. If both $S^k$'s have trivial normal bundle, then $M$ is obtained by gluing two copies of $S^k \times \mathbb{D}^{n-k} $ along their boundaries $S^k \times S^{n-k-1} = \partial (S^k \times \mathbb{D}^{n-k} )$ with a self-diffeomorphism.
\end{theorem*}

\noindent The proof of the above for a Morse-Bott function is standard (see \autoref{LMZ}). 


\vspace{0.3cm}
\hspace{0.3cm}Hai Bao and Rees \cite{HaiRe92} analyzed the case of smooth function on a compact manifold $M$ whose critical set is the disjoint union of a point and a connected smooth submanifold of positive dimension. They proved (\cite{HaiRe92} Theorem (1)) that $M$ has the cohomology ring structure of a projective space and the cohomology of the critical submanifold corresponds to a codimension one projective subspace. There is no assumption on the function $f$ being Morse-Bott. We prove a variant of \autoref{thm:introSphere}, adapted to projective spaces. This also generalizes the result of Hai Bao and Rees in the context of Morse-Bott functions. As a combination of \Cref{thm:forComplexProjective,thm:forRealProjective,thm:forQuaternionicProjective} we have the following result.

\begin{thmIntro}\label{thm:introProjective}
    Let $M$ be a closed smooth manifold of dimension $d$. Let $\mathbb{K} = \mathbb{R}, \mathbb{C}, \mathbb{H}$ and let $\mathbb{KP}^n$ denote the projective space over $\mathbb{K}$ of real dimension $n\dim_{\mathbb{R}} \mathbb{K}$. Let $f$ be a Morse-Bott function on $M$ with exactly two critical submanifolds, $\mathbb{KP}^k$ and $\mathbb{KP}^l$ such that $k < l$. Then the dimension of $M$ is $d =(k + l + 1)\dim_{\mathbb{R}} \mathbb{K}$, and $M$ is homotopy equivalent to $\mathbb{KP}^{k+l+1}$.
\end{thmIntro}
The method of proof for the complex and quaternionic cases (\autoref{thm:forComplexProjective} and \autoref{thm:forQuaternionicProjective} respectively) are identical and do not use \autoref{thm:introSphere}. The case of real projective spaces (\autoref{thm:forRealProjective}) requires a different method and crucially uses \autoref{thm:introSphere}. It is important to note that the dimension of $M$ is a consequence of the hypothesis in both \autoref{thm:introSphere} and \autoref{thm:introProjective}. For instance, a slightly different proof of \autoref{thm:introProjective} for the real case, with the added assumption of $d=k+l+1$ while dropping the requirement of non-degeneracy of $f$ at the critical sets, is given in \cite{HaiRe92}, Corollary 2 following Theorem (2) on page 144. As in Theorem A, without $k=l$, Theorem B fails, i.e., $M$ need not be of required dimension or homotopy equivalent to a projective space, even if the dimension is assumed to be the correct one. Moreover, the homotopy equivalence cannot be improved to homeomorphism.

\section{The case of spheres}
\begin{theorem}\label{thm:forSpheres}
    Let $M$ be a closed, smooth manifold of dimension $d$. Let $f:M\to \mathbb{R}$ be a Morse-Bott function with two connected, critical submanifolds $S^k$ and $S^l$, of unequal dimension. Then $d = k+l+1$ and $M$ is homeomorphic to $S^d$.
\end{theorem}

\noindent We note that $k\neq l$ is a necessary condition. For example, consider a horizontally placed torus $T^2$ in $\mathbb{R} ^3$ (see \autoref{fig:torusOnHorizontalPlane}). The height function
    \[
        h:T^2 \to \mathbb{R}, (x,y,z) \mapsto z
    \] 
    is a Morse-Bott function with critical submanifolds $S_{\text{T}}^1$, where $h$ attains its maximum, and $S_{\text{B}}^1$, where $h$ attains its minimum.
    \begin{figure}[H]
        \centering
    \def\svgwidth{0.85\columnwidth}
    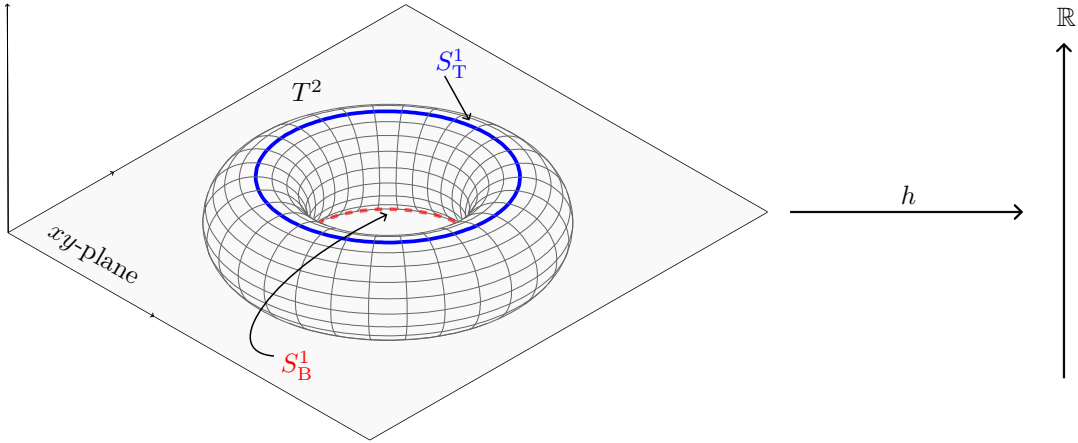

        \caption{Height function on the torus that is placed horizontally}
        \label{fig:torusOnHorizontalPlane}
    \end{figure}

\hspace*{0.3cm} The following example shows the existence of such Morse-Bott function on the standard sphere with two spheres of unequal dimensions as critical submanifolds. It was motivated by the analysis of the cut locus of $S^k \hookrightarrow S^{k+l+1} $(see \cite[Example 2.7]{BaPr23}).
\begin{eg}\label{eg:morseBottFunctionOnSphere}
    Define a function 
    \[
        f: S^{k+l+1} \to  \mathbb{R} , \ (x_0, x_1, \dots, x_{k+l+1}) \mapsto -\sum_{i=0}^{k} x_i^2 + \sum_{i=k + 1}^{k + l + 1} x_i^2 = 1 - 2 \sum_{i=0}^{k} x_i^2
    \] 
    Then $f$ is a Morse-Bott function with critical submanifolds $S^k = f^{-1} (-1)$ and $S^l = f^{-1} (1)$. This example is a geometric realization of $S^{k+l+1}$ as the topological join of $S^k$ and $S^l$.
\end{eg}
\hf The following lemma is standard in Morse theory and follows by considering the negative gradient flow. We omit the proof.
\begin{lemma}\label{lem:morseDiffeo}
    Let $f: M \to \mathbb{R} $ be a smooth function on a manifold $M$. Let $[a,b] \subset \mathbb{R} $ be an interval which does not contain any critical values of $f$. If $f^{-1}[a,b]$ is compact, then for any $c \in [a,b]$, $f^{-1} [a,b]$ is diffeomorphic to $f^{-1} (c) \times [a,b]$.    
\end{lemma}
%

\begin{proof}[Proof of \autoref{thm:forSpheres}]
    Let $k<l$. By an affine transformation of $\mathbb{R}$, we may assume that $f$ takes values in $[-1,1]$ as well as $S^k=f^{-1}(-1)$ and $S^l=f^{-1}(1)$. Note that if $l=d$, then $M$ will have one path component $S^l$ on which $f$ is identically $1$ and at least one more component $M'$, where $f|_{M'}$ is Morse-Bott with one global minima $S^k$. The global maxima on $M'$ is also a critical submanifold, which is a contradiction, unless $f$ is constant on $M'$. This forces $k=d=l$, which is not possible. Thus, we conclude that $l<d$.\\
    \begin{figure}[!htb]
        \centering
    \def\svgwidth{0.3\columnwidth}
\begingroup%
  \makeatletter%
  \providecommand\color[2][]{%
    \errmessage{(Inkscape) Color is used for the text in Inkscape, but the package 'color.sty' is not loaded}%
    \renewcommand\color[2][]{}%
  }%
  \providecommand\transparent[1]{%
    \errmessage{(Inkscape) Transparency is used (non-zero) for the text in Inkscape, but the package 'transparent.sty' is not loaded}%
    \renewcommand\transparent[1]{}%
  }%
  \providecommand\rotatebox[2]{#2}%
  \newcommand*\fsize{\dimexpr\f@size pt\relax}%
  \newcommand*\lineheight[1]{\fontsize{\fsize}{#1\fsize}\selectfont}%
  \ifx\svgwidth\undefined%
    \setlength{\unitlength}{557.33283576bp}%
    \ifx\svgscale\undefined%
      \relax%
    \else%
      \setlength{\unitlength}{\unitlength * \real{\svgscale}}%
    \fi%
  \else%
    \setlength{\unitlength}{\svgwidth}%
  \fi%
  \global\let\svgwidth\undefined%
  \global\let\svgscale\undefined%
  \makeatother%
  \begin{picture}(1,0.88405193)%
    \lineheight{1}%
    \setlength\tabcolsep{0pt}%
    \put(0,0){\includegraphics[width=\unitlength,page=1]{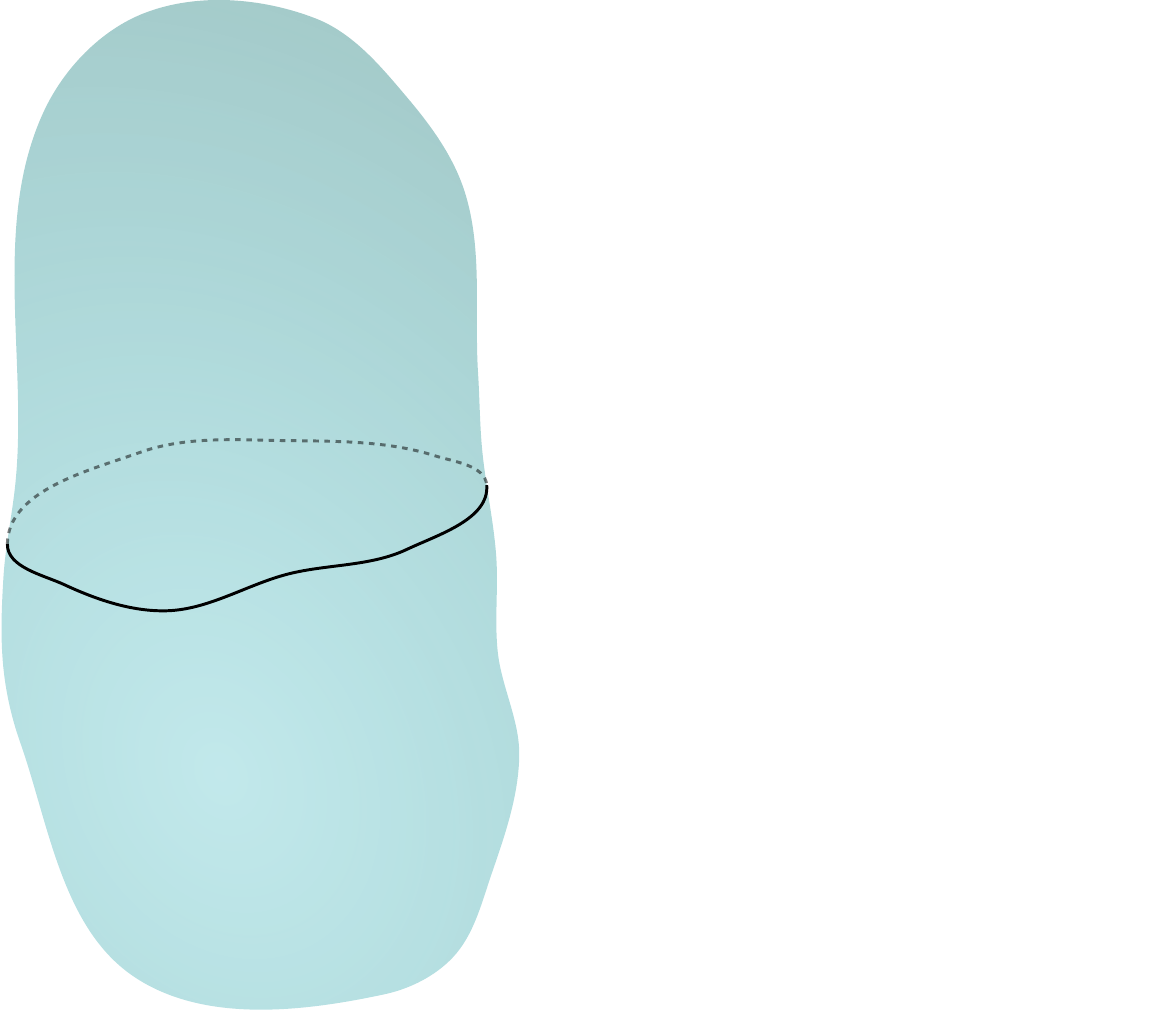}}%
    \put(0.43826095,0.87518496){\color[rgb]{0,0,0}\transparent{0.92549002}\makebox(0,0)[t]{\lineheight{1.25}\smash{\begin{tabular}[t]{c}$S^l=f^{-1}(1)$\end{tabular}}}}%
    \put(0.07368468,0.29005779){\color[rgb]{0,0,0}\transparent{0.92549002}\makebox(0,0)[t]{\lineheight{1.25}\smash{\begin{tabular}[t]{c}$M_0=f^{-1}(0)$\end{tabular}}}}%
    \put(0.51194493,0.01928909){\color[rgb]{0,0,0}\makebox(0,0)[t]{\lineheight{1.25}\smash{\begin{tabular}[t]{c}$S^k = f^{-1}(-1)$\end{tabular}}}}%
    \put(0,0){\includegraphics[width=\unitlength,page=2]{morse-bott-function-on-M.pdf}}%
    \put(0.98763501,0.88481847){\color[rgb]{0,0,0}\makebox(0,0)[t]{\lineheight{0}\smash{\begin{tabular}[t]{c}$\mathbb{R}$\end{tabular}}}}%
    \put(0.73708751,0.41239956){\color[rgb]{0,0,0}\makebox(0,0)[t]{\lineheight{0}\smash{\begin{tabular}[t]{c}$f$\end{tabular}}}}%
    \put(0,0){\includegraphics[width=\unitlength,page=3]{morse-bott-function-on-M.pdf}}%
  \end{picture}%
\endgroup%

        \caption{Morse-Bott function $f$ \label{fig:Morse-Bott-function-on-M}}
    \end{figure}
    \hspace*{0.5cm} Choose a Riemannian metric $g$ on $M$ and consider the normalized gradient flow $X=-(\nabla f)/\|\nabla f\|$ on $f^{-1}[a,b]$, where $-1<a<b<1$. Then, from Morse-Bott theory, we have the following:\\
    \hspace*{0.5cm} (i) \autoref{lem:morseDiffeo} implies that $f^{-1}[a,b]$ is diffeomorphic to $M_0\times [a,b]$, whenever $-1<a<0<b<1$. Moreover, $f^{-1}(c)$ is diffeomorphic to $M_0$ for any $c\in [a,b]$;\\
    \hspace*{0.5cm} (ii) There exists $\epsilon>0$ such that $M^{-1+\epsilon}:=f^{-1}(-\infty,-1+\epsilon]$ is homeomorphic to the unit disk bundle $D(\nu_{S^k})$ of the normal bundle $\nu_{S^k}$ of $S^k$ in $M$;\\
    \hspace*{0.5cm} (iii) There exists $\delta>0$ such that $f^{-1}[1-\delta,\infty)$ is homeomorphic to the unit disk bundle $D(\nu_{S^l})$ of the normal bundle $\nu_{S^l}$ of $S^l$ in $M$.\\
    It follows from (ii) and (iii) that $M_0$ is diffeomorphic to the unit sphere bundle of $\nu_{S^k}$ as well as that of $\nu_{S^l}$.    \\
    Consider the sphere bundles $\mathcal{F}_k$ and $\mathcal{F} _l$ over $S^k$ and $S^l$ respectively, i.e.,  
    \[
        \left.
            \begin{tikzcd}
                S^{d-k-1} \arrow[r,hook] & M_0 \arrow[d]\\
                                                & S^k
            \end{tikzcd}
            \kern 0.2cm 
        \right\} \mathcal{F}_k
        \kern 1cm
        \left.
            \begin{tikzcd}
                S^{d-l-1} \arrow[r,hook] & M_0 \arrow[d]\\
                                               & S^l
            \end{tikzcd}
            \kern 0.2cm 
        \right\} \mathcal{F}_l
    \]  
    We have introduced the key ingredients of the proof. Note that $f^{-1}[0,1]$ and $f^{-1}[-1,0]$ are homotopy equivalent to $S^l$ and $S^k$ respectively. As $l>1$, $f^{-1}[0,1]$ is simply connected and as $k\geq 1$, $f^{-1}[-1,0]$ is path connected. Since $M_0$ is non-empty, the union $f^{-1}[-1,0]\cup f^{-1}[0,1]=M$ is connected. \\ 
    \hspace*{0.5cm}We will prove our result as a combination of the following steps: 
    \begin{enumerate}[(1)]
        \item[] \textbf{Step 1:} $M$ is of dimension $k + l + 1$ (\autoref{M-dim}).
        \item[] \textbf{Step 2:} $M$ is simply connected (\autoref{M-pi1=0}).
        \item[] \textbf{Step 3:} $M$ is a homology sphere (\autoref{M-hom-sph}). 
    \end{enumerate}
    Assuming these steps, by the Hurewicz theorem, we have $\pi _d(M) \cong H_d(M) \cong \mathbb{Z} $. We may invoke the following classical and important results:
    \begin{enumerate}
        \item If $d = 2$, then $M$ is a smooth sphere by classification of oriented closed surfaces and uniqueness of smooth structures.
        \item If $d = 3$, then $M$ is a smooth sphere by Perelman's proof of the Poincar\'{e} conjecture \cite{Perelman} and uniqueness of smooth structures.
        \item If $d = 4$, then by Freedman's result \cite{Fre82}, $M$ is a topological sphere. 
        \item If $d\geq 5$, then by generalized Poincar\'{e} conjecture \cite{Sma61} due to Smale, $M$ is a topological sphere.   
    \end{enumerate} 
    Hence, the theorem is proved. 
\end{proof}
    As indicated in the introduction, \autoref{thm:forSpheres} is sharp in the sense that the homeomorphism cannot be extended to a diffeomorphism. 
    \begin{lemma}\label{M-dim}
        With the conditions as in \autoref{thm:forSpheres}, the dimension of $M$ is $k+l+1$.
    \end{lemma}
    \begin{proof} As $k\geq 1$, this forces $l\geq 2$, whence $d \geq l + 1 \geq 3$. We now make the following two claims:
    \begin{enumerate}[(1)]
        \item[] \textbf{Claim 1:}  $d - k > k$; \label{claim-1}
        \item[] \textbf{Claim 2:}  $d - l \leq l$. \label{claim-2}
    \end{enumerate}
    \begin{proof}[Proof of claim 1]
        If the claim is false, then
        \[
            d - l < d-k \leq k < l.
        \] 
        Using the long exact sequence of homotopy groups for the fibration $\mathcal{F} _l$ and $\mathcal{F} _k$  we get 
        \begin{equation}\label{eqn:longExactSequenceFl}
            \mathcal{F}_l: \kern 0.3cm \cdots \to \pi_{d-l}(S^l) \to \pi_{d - l - 1} ( S^{d - l - 1} ) \rightarrow \pi_{d - l - 1} \left( M_0 \right) \to  \pi_{d - l - 1}( S^l) \to  \cdots
        \end{equation}
        \begin{equation}\label{eqn:longExactSequenceFk}
            \mathcal{F}_k: \kern 0.3cm \cdots \to \pi _{d-l}(S^k) \to  \pi _{d-l-1}(S^{d - k - 1}) \to  \pi _{d-l-1}(M_0) \to \pi _{d-l-1}(S^k) \to  \cdots
        \end{equation}
        If $d- l - 1 > 0$, then using \eqref{eqn:longExactSequenceFl} we get 
        \[
            \mathbb{Z} \cong \pi _{d-l-1}(S^{d-l-1}) \cong \pi _{d-l-1}(M_0).
        \]  
        On the other hand, using \eqref{eqn:longExactSequenceFk}, we get
        $\pi _{d-l-1}(M_0) \cong 0$, which is a contradiction, implying $d=l+1$.\\
        \hspace*{0.5cm}        If $d = l + 1$, then the normal bundle $\nu_{S^l}$ is a line bundle over $S^l$. As $S^l$ is simply connected, $\nu_{S^l}$ is trivializable, and $M_0\cong S^l\sqcup S^l$. 
        The sequence \eqref{eqn:longExactSequenceFk} reduces to
        \begin{equation*}
            \cdots \to \pi_1(S^k)\to \pi_{0} ( S^{l-k} ) \rightarrow \pi_{0} \left( M_0 \right) \to  \pi_{0}( S^k) \to 0,
        \end{equation*}
        which is again not possible as $l-k \geq 1$. This forces $d - k > k$.  
    \end{proof}
    As $d-k >k\geq 1$, the fibre of $\mathcal{F} _k$ is always connected and hence $M_0$ is also connected.

    \begin{proof}[Proof of claim 2]
        If the claim is false, then
        \[
             d - k > d - l > l > k \implies d - l \geq k + 2 \text{ and } d-k \geq k +3.
        \]
        Now consider the long exact sequence of homotopy groups for the fibrations $\mathcal{F} _l$ and $\mathcal{F}_k$:
        \begin{equation}\label{eq:longExactSequenceFlClaim2}
            \begin{tikzcd}[column sep=2em]
                \mathcal{F}_l : \kern 0.3cm \cdots \arrow[r] & 
                \cancelto{0}{\pi_k(S^{d-l-1})} \arrow[r] & 
                \pi_k(M_0) \arrow[r] & 
                \cancelto{0}{\pi_k(S^l)} \arrow[r] & 
                \pi_{k-1}(S^{d-l-1}) \arrow[r] & 
                \cdots
            \end{tikzcd}
        \end{equation}
        \begin{equation}\label{eq:longExactSequenceFkClaim2}
            \begin{tikzcd}[column sep=2em]
                \mathcal{F}_k : \kern 0.3cm \cdots \arrow[r] & 
                \cancelto{0}{\pi_k(S^{d-k-1})} \arrow[r] & 
                \pi_k(M_0) \arrow[r] & 
                \pi_k(S^k) \arrow[r] & 
                \pi_{k-1}(S^{d-k-1}) \arrow[r] & 
                \cdots
            \end{tikzcd}
        \end{equation}
        Using \eqref{eq:longExactSequenceFlClaim2}, we see that $\pi _k(M_0) = 0$. On the other hand, if $k > 1$, then \eqref{eq:longExactSequenceFkClaim2} implies $\pi _k (M_0) \cong \pi _k (S^k) \cong \mathbb{Z} $, a contradiction. If $k = 1$, then \eqref{eq:longExactSequenceFkClaim2} becomes 
        \[
            \cdots \rightarrow 0 \rightarrow  \pi _1(M_0) \rightarrow \mathbb{Z} \rightarrow \pi _0(S^{d-k-1}) \rightarrow \pi _0(M_0) \rightarrow \pi _0(S^k) \to 0.
        \]
        Since $S^{d-k-1}, M_0$ and $S^k$ are connected, we get a short exact sequence of pointed sets and thus, $\pi _1(M_0) \cong \mathbb{Z} $, which is a contradiction. This forces $d - l \leq l$. 
    \end{proof}
        The long exact sequence for $\mathcal{F}_k$ along with \hyperref[claim-1]{claim 1} implies that $M_0$ is $(k-1)$-connected and $\pi_k(M_0)\neq 0$. The long exact sequence for $\mathcal{F}_l$ along with \hyperref[claim-2]{claim 2} implies that  $M_0$ is $(d- l - 2)$-connected. Therefore, $d - l - 2 \leq k - 1$ which implies $d \leq k + l + 1$. Note that $\pi_{d-l-1}(M_0)\neq 0$ if and only if $d-l-1=k$. Look at the long exact sequence for $\mathcal{F}_l$
        \begin{align*}
            \mathcal{F} _l: & \kern 0.3cm \cdots \to \pi _{d-l}(S^l) \to  \pi _{d-l-1}(S^{d-l-1}) \to  \pi_{d-l-1}(M_0) \to \pi_{d-l-1}(S^l)\to \cdots
        \end{align*}  
        If $\pi_{d-l-1}(M_0)= 0$, then $\pi_{d-l} (S^l)$ surjects onto $\mathbb{Z} $. This implies, by Serre's result on homotopy groups of spheres, that one of the following mutually exclusive possibilities must hold:\\
        \hspace*{0.5cm} (i) $d - l = l$, or\\
        \hspace*{0.5cm} (ii) $d - l = 2l - 1$ and $l$ is even.\\
        Since, $d \leq k + l + 1 < l + l + 1 = 2l + 1$, so $d \leq 2l$. Thus, $d = 3l - 1$ is not possible. If $d = 2l$, then 
        \[
            d < k + l + 1 \implies l < k + 1,
        \]         
        which is impossible as $k < l$. This completes the proof of the lemma.
    \end{proof}

    \begin{lemma}\label{M-pi1=0}
        With the conditions as in \autoref{thm:forSpheres}, $M$ is simply connected.
    \end{lemma}
    \begin{proof}
    Consider the normal bundle $\nu_{S^k}\to S^k$ of rank $d-k=l+1$. As the rank is greater than the dimension of the base, we choose a splitting 
    \begin{equation}
        \nu_{S^k}\cong \epsilon^{l+1-k}\oplus \xi^k,\label{split}
    \end{equation}
    where $l+1-k\geq 2$. In particular, the fibre bundle $\mathcal{F}_k$ has a section $s$. Thus, there is an isomorphism
    \begin{displaymath}
    \varphi: \pi_j(S^k)\oplus \pi_j(S^l)\xrightarrow{s_\ast + i_\ast} \pi_j(M_0).    
    \end{displaymath}
    In particular, when $j=1$, the map $s_\ast:\pi_1(S^k)\to \pi_1(M_0)$ is an isomorphism. Let $\iota:M_0\hookrightarrow \nu_{S^k}$ be the inclusion and let $\pi:\nu_{S^k}\to S^k$ be the bundle map. Then the map $\pi\,{\scriptstyle\circ}\,\iota:M_0\to S^k$ satisfies 
    \begin{displaymath}
        \pi\,{\scriptstyle\circ}\,\iota\,{\scriptstyle\circ}\,s = \textup{id}_{S^k}.
    \end{displaymath}
    As $s_\ast$ is an isomorphism, it follows that $\pi_\ast\,{\scriptstyle\circ} \,\iota_\ast$ is also an isomorphism.\\
    \hspace*{0.5cm} Let us consider the open cover $U=f^{-1}[-1,1/2)$ and $V=f^{-1}(-1/2,1]$ of $M$. Note that $U$ deforms to $D(\nu_{S^k})\simeq S^k$ and $V$ deforms to $D(\nu_{S^l})\simeq S^l$ respectively, with $U\cap V$ having $M_0$ as a deformation retract. If $k=1$, then $U, V$ and $U\cap V$ are path-connected; Seifert-van Kampen Theorem applied to the cover $\{U,V\}$ imply that 
    \begin{displaymath}
        \pi_1(M)\cong\pi_1(U)\ast_{\pi_1(M_0)}\pi_1(V)\cong \pi_1(\nu_{S^k})/\iota_\ast \pi_1(M_0)\cong \pi_1(S^k)/\pi_\ast\iota_\ast \pi_1(M_0)=0.
    \end{displaymath}
    If $k>1$, then $\pi_1(U)=\pi_1(V)=0$ implies that $\pi_1(M)=0$.
    \end{proof}

    \begin{lemma}\label{M-hom-sph}
        With the conditions as in \autoref{thm:forSpheres}, $M$ is a homology sphere.
    \end{lemma}
    \begin{proof}
    Recall the splitting \eqref{split}
    \begin{displaymath}
    \nu_{S^k}\cong \epsilon^{l+1-k}\oplus \xi^k.
    \end{displaymath}
    As $S^k$ and $M$ are oriented, the normal bundle $\nu_{S^k}$ and the sphere bundle $M_0$ are oriented. Thus, the action of $\pi_1(S^k)$ on the homology of the fibre sphere is trivial. As this fibre bundle has a section, it follows from the Serre spectral sequence that 
    \begin{equation}
         H_\bullet(M_0) \cong H_\bullet ( S^k ) \otimes H_\bullet ( S^l) .\label{M_0}
    \end{equation}
    Let $U$ and $V$ be as in the proof of \autoref{M-pi1=0}. Using the Mayer-Vietoris sequence, for any $i<k$, we have 
    \[
        \begin{tikzcd}[column sep = 2em]
            \cdots \arrow[r] & \cancelto{0}{H_i(M_0)} \arrow[r] & \cancelto{0}{H_i(U) \oplus H_i(V)} \arrow[r] & H_i(M) \arrow[r] & \cancelto{0}{H_{i-1}(M_0)} \arrow[r] & \cdots
        \end{tikzcd}
    \]
    Thus $H_i(M) \cong 0$ for $i < k$. Also, if $i = d-1$, then we have
    \[
        \begin{tikzcd}[column sep = 2em]
            \cdots \arrow[r] & 0 \arrow[r] & H_d(M) \arrow[r] & H_{d-1}(M_0) \arrow[r] & 0 \arrow[r] &  \cdots,
        \end{tikzcd}
    \]
    which implies $H_d(M) \cong \mathbb{Z} $. 
     
    \hspace{0.3cm} Look at the $k^{\text{th} }$ homology of $M$ in the Mayer-Vietoris sequence:
    \[
        \begin{tikzcd}[column sep = 2em]
            \cdots \arrow[r] & H_k(M_0) \arrow[r] & H_k(U) \oplus \cancelto{0}{H_k(V)} \arrow[r] & H_k(M) \arrow[r] & \cancelto{0}{H_{k-1}(M_0)} \arrow[r] & \cdots 
        \end{tikzcd}
    \] 
    Since $d-k > k$, the fibration $\mathcal{F} _k$ has a section $s$. Note that $s(S^k) \hookrightarrow M_0$ generates $H_k(M_0)$. Also, $s$ is homotopic to the zero section of the bundle. Thus, the two copies of $S^k$ inside $U$ (one $S^k$ is the zero section which generates $H_k(U)$ and other is $s(S^k)$) are homotopic. This implies that $H_k(M_0)\to H_k(U)$ is an isomorphism, whence $H_k(M) \cong 0$. Also, $H_i(M) \cong 0$ for $k \leq i \leq l-1$. Now look at the sequence for $l = k+1$.
    \[
        \cdots \longrightarrow H_l(M_0) \longrightarrow H_l(U) \oplus H_l(V) \longrightarrow H_l(M) \longrightarrow H_k(M_0) \longrightarrow \cdots, 
    \] 
    which gives 
    \[
        \cdots \longrightarrow \mathbb{Z} \longrightarrow \mathbb{Z} \longrightarrow   H_l(M) \longrightarrow   \mathbb{Z} \longrightarrow \cdots. 
    \]
    Thus, $H_l(M)$ can only have torsion. If $\tau H_i (M)$ denotes the torsion part of $H_i(M)$, then by Poincar\'{e} (or linking) duality, we have 
    \[
        \tau H_l(M) \cong \tau H_{d-l-1}(M) = \tau H_k(M) \cong 0.
    \]
    This establishes that $H_l(M) \cong 0$. It is clear that $H_i(M) \cong 0$ for $l\leq i \leq d-1$. Thus, $M$ is a homology sphere of dimension $d$. 
    \end{proof}
\begin{remark}
    It is natural to ask, based on \eqref{M_0}, whether $M_0$ is homeomorphic (or even homotopic) to $S^k \times S^l$. It follows from a result \cite[Theorem 6.10]{BaBhSa24} of Basu-Bhowmick-Samanta that $M_0$ is rationally homotopy equivalent to $S^k \times S^l$. Unless further hypothesis is imposed, we do not expect $M_0$ to be homotopy equivalent to $S^k\times S^l$.
\end{remark}

\begin{eg}\label{eg:ht-square-v2}
    In the hypothesis of \autoref{thm:forSpheres}, we may drop the connectivity assumption on $S^k$, i.e., assume that $k=0<l$ with $S^0=\{p,q\}$. There are (up to natural symmetry) four cases and the first three cases cannot occur:\\
    \hspace*{0.5cm}(i) $f(p)=f(S^l)=f(q)$: As $f$ is constant on the critical set, this forces the maxima and minima of $f$ to be the same, whence it is constant.\\
    \hspace*{0.5cm}(ii) $f(p)>f(S^l)=f(q)$: Suppose $1=f(p)$ and $-1=f(S^l)=f(q)$. Then $f^{-1}(0)$ is diffeomorphic to the the boundary of the normal bundle of $\{p\}$ as well as that of $S^l\sqcup \{q\}$. The former is $S^{d-1}$ while the latter, $S(\nu_{S^l})\sqcup S^{d-1}$, is disconnected. \\
    \hspace*{0.5cm}(iii) $f(p)>f(S^l)>f(q)$: Suppose $1=f(p), 0=f(S^l), -1=f(q)$. Then the boundary of an appropriate tubular neighbourhood of $S^l$ is identified with $f^{-1}\{\tfrac{1}{2},-\tfrac{1}{2}\}$, which is homeomorphic to $S^{d-1}\sqcup S^{d-1}$. This forces $l=d-1$ and the tubular neighbourhood to be trivial. As $S^l$ is a critical submanifold of codimension $1$ of a Morse-Bott function $f$, this function looks like $\pm x^2$ in the tubular neighbourhood of $S^l$, where $x$ is the normal co-ordinate. This forces $f$ to be decreasing or increasing in the normal direction of $S^l$. As the global maxima and minima are at $p$ and $q$, this forces the existence of further critical points other than $\{p,q\}\sqcup S^l$, a contradiction.   \\
    \hspace*{0.5cm}(iv) $f(p)\geq f(q)>f(S^l)$: Suppose $1=f(p)\geq f(q)=c>f(S^l)=0$.Then $M_{c/2} \simeq S^{d-1}\sqcup S^{d-1}$ is disconnected. Using the fibration $\mathcal{F} _l$, we have
    \[
        \pi _1(S^l) \rightarrow \pi _0(S^{d-l-1}) \rightarrow \pi _0(M_0) \rightarrow \pi _0 (S^l) \rightarrow 0.
    \]  
    Since $\pi _0(M_0)$ contains two points and $\pi _0(S^l)$ is singleton, the set $\pi _0(S^{d-l-1})$ must contain at least two points. Thus, $S^{d-l-1}$ must be disconnected and hence, $d = l+1$. In fact, any line bundle over $S^l$ is trivializable as $S^l$ is simply connected. Thus, $M$ is obtained by gluing $\mathbb{D}^d\sqcup \mathbb{D}^d$ along their boundary to $S^l\times [-1,1]$. This provides a homeomorphism between $M$ and $S^d$. \\
    \hspace*{0.5cm}If we consider the square of the height function $h^2:S^{d+1}\to\mathbb{R}$, then this is Morse-Bott with critical set consisting of the two poles (together considered as the $0$-sphere $S^0$) and the equator 
    $$S^d\cong \{(x_1,\ldots,x_{d+2})\in S^{d+1}\,|\,x_{d+2}=0\}.$$
    More generally, if $c\in (0,1)$, then choose an increasing smooth function $\rho_c:\mathbb{R}\to \mathbb{R}$ satisfying\\
    \hspace*{0.3cm} (a) $\rho_c(x)=x$ for $x\leq \tfrac{c}{2}$;\\
    \hspace*{0.3cm} (b) $\rho_c(1)=\sqrt{c}$;\\
    We may consider the Morse-Bott function $(\rho_c \,{\scriptstyle\circ}\,h)^2:S^{d+1}\to\mathbb{R}$. This is the typical function satisfying case (iv).    
    \begin{figure}[ht]
    	\centering
    	\def\svgwidth{0.4\columnwidth}
    	\graphicspath{{figures/}}
\begingroup%
  \makeatletter%
  \providecommand\color[2][]{%
    \errmessage{(Inkscape) Color is used for the text in Inkscape, but the package 'color.sty' is not loaded}%
    \renewcommand\color[2][]{}%
  }%
  \providecommand\transparent[1]{%
    \errmessage{(Inkscape) Transparency is used (non-zero) for the text in Inkscape, but the package 'transparent.sty' is not loaded}%
    \renewcommand\transparent[1]{}%
  }%
  \providecommand\rotatebox[2]{#2}%
  \newcommand*\fsize{\dimexpr\f@size pt\relax}%
  \newcommand*\lineheight[1]{\fontsize{\fsize}{#1\fsize}\selectfont}%
  \ifx\svgwidth\undefined%
    \setlength{\unitlength}{492.27283868bp}%
    \ifx\svgscale\undefined%
      \relax%
    \else%
      \setlength{\unitlength}{\unitlength * \real{\svgscale}}%
    \fi%
  \else%
    \setlength{\unitlength}{\svgwidth}%
  \fi%
  \global\let\svgwidth\undefined%
  \global\let\svgscale\undefined%
  \makeatother%
  \begin{picture}(1,0.59213607)%
    \lineheight{1}%
    \setlength\tabcolsep{0pt}%
    \put(0,0){\includegraphics[width=\unitlength,page=1]{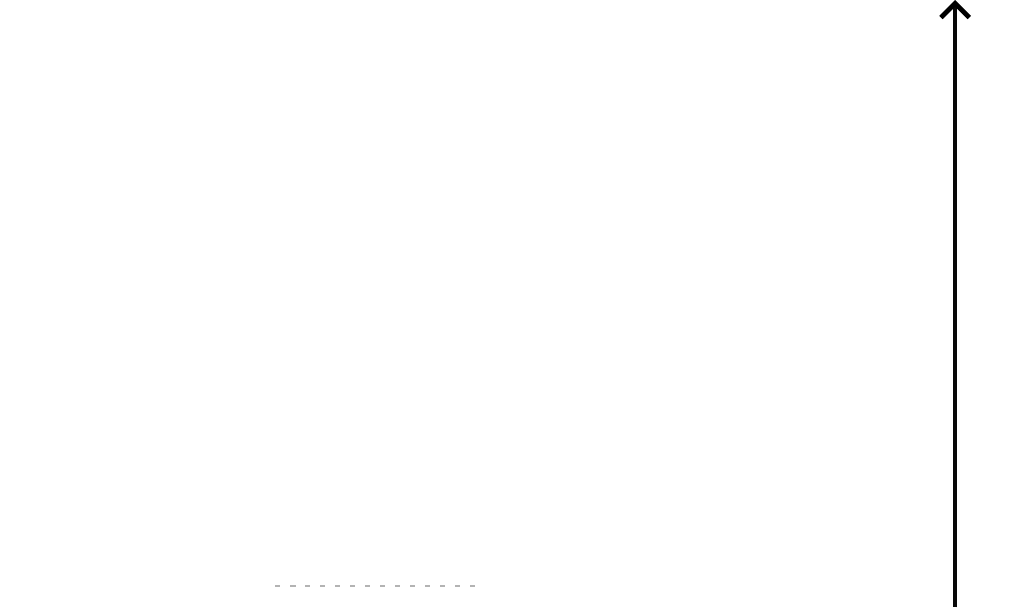}}%
    \put(0.9675113,0.57495271){\color[rgb]{0,0,0}\makebox(0,0)[t]{\lineheight{0}\smash{\begin{tabular}[t]{c}$\mathbb{R}$\end{tabular}}}}%
    \put(0,0){\includegraphics[width=\unitlength,page=2]{height-squared-function.pdf}}%
    \put(0.79216826,0.33365159){\color[rgb]{0,0,0}\makebox(0,0)[t]{\lineheight{0}\smash{\begin{tabular}[t]{c}$h^2$\end{tabular}}}}%
    \put(0,0){\includegraphics[width=\unitlength,page=3]{height-squared-function.pdf}}%
    \put(0.89607938,0.54338261){\color[rgb]{0,0,0}\makebox(0,0)[t]{\lineheight{0}\smash{\begin{tabular}[t]{c}$1$\end{tabular}}}}%
    \put(0.89912652,0.05127754){\color[rgb]{0,0,0}\makebox(0,0)[t]{\lineheight{0}\smash{\begin{tabular}[t]{c}$0$\end{tabular}}}}%
    \put(0,0){\includegraphics[width=\unitlength,page=4]{height-squared-function.pdf}}%
  \end{picture}%
\endgroup%

    	\caption{height-squared function} 
    	\label{fig:height-squared-function}
    \end{figure}
\end{eg}
    \begin{remark}\label{LMZ}
    With the hypothesis as in \autoref{thm:forSpheres}, if we further assume that the normal bundles $\nu_{S^k}$ and $\nu_{S^l}$ are trivializable, then $M_0$ is diffeomorphic to $S^k\times S^l$. The manifold $M$ is obtained by gluing $S^k\times \mathbb{D}^{l+1}$ with $\mathbb{D}^{k+1}\times S^l$ along $M_0$. 
    \end{remark}

    \begin{eg}\label{eg:sigma7-as-join}
    In \cite[Theorem 1.1 and \S 2]{DuPu2008}, Dur\'{a}n and P\"{u}ttmann has shown that the generating exotic sphere $\Sigma^7$ (also realized as a biquotient of the Lie group $Sp(2)$ by Gromoll-Meyer \cite{GrM74}) is the geodesic join of a geodesic loop $S^1$ and a minimal subsphere $\Sigma^5$, with the distance between $S^1$ and $\Sigma^5$ being $\pi/2$. This, in turn, allows us to define a Morse-Bott function on $\Sigma^7$ with $S^1\sqcup \Sigma^5$ as the critical set. This illustrates that the conclusion of $M$ being homeomorphic to $S^d$ in \autoref{thm:forSpheres} cannot be upgraded to being diffeomorphic.
    \end{eg}

    \hspace*{0.5cm}We may assume that $f:M\to \R$ is Morse-Bott with $S^k=f^{-1}(-1)$ and $S^l=f^{-1}(1)$ are critical submanifolds of $f$ but impose no conditions on $k$ and $l$. As $k\neq l$ is covered in Theorem \autoref{thm:forSpheres}, this leaves the case $k=l$. Without a reasonable assumption like $d=k+l+1=2k+1$, even this case will not work generically. Consider any sphere bundle $\pi:E\to S^m$ with fibre $S^k$. If $h:S^m\to\R$ denotes the height function, then $h\,{\scriptstyle\circ}\,\pi$ is Morse-Bott with critical set $S^k\sqcup S^k$. We now explore the scenario when $k=l$ and $d=2k+1$. 
    \begin{eg}\label{eg:k=1}
        When $k=1$, we are studying Morse-Bott functions $f:M^3\to \R$ with critical set $S^1\sqcup S^1$. Note that if $\gamma\to S^2$ is the tautological (complex) line bundle over $S^2=\mathbb{CP}^1$, then the sphere bundles $P_n:=S(\gamma^{\otimes n})$ represent all the distinct $S^1$-bundles over $S^2$, as $n$ varies over the integers. Composing the bundle map with the height function on $S^2$ will give Morse-Bott functions with the required property. However, $P_n$'s (for $n>0$) are not homotopy equivalent. This follows, for instance, by the fact that $\pi_1(P_n)\cong\Z/n\Z$. Note that $(P_n)_0$ is diffeomorphic to the torus.
    \end{eg}
    \begin{eg}\label{eg:k=2}
        When $k=2$, we are dealing with Morse-Bott functions $f:M^5\to \R$. Note that $M_0$, being a $S^2$-bundle over $S^2$, is simply connected. It follows that $M$ is simply connected. There are two possibilities for $M_0$: $S^2\times S^2$ and $\mathbb{CP}^2\#\overline{\mathbb{CP}}^2$. If the critical set of $f$ is $S^2_1\sqcup S^2_2$, then we apply Mayer-Vietoris sequence (in homology) to the cover $\{M\setminus S^2_1, M\setminus S^2_2\}$ to obtain
        $$0\longrightarrow H_3(M)\longrightarrow H_2(M_0)\stackrel{i}{\longrightarrow} H_2(M\setminus S^2_1)\oplus H_2(M\setminus S^2_2)\longrightarrow H_2(M)\longrightarrow 0$$
        Note that $M\setminus S^2_i$ deforms to $S^2_j$, $i\neq j$. Thus, the two middle groups are both $\mathbb{Z}^2$. This forces $H_3(M)$ to be free of rank $0,1$ or $2$. As the normal bundle $\nu_l$ of $S^2_l$ ($l=1,2$) is of rank $3$, both the sphere bundles $M_0=S(\nu_1)\cong S(\nu_2)$ admit sections. Let $s_l:S^2_l\to M_0$ denote these sections. Then $(s_l)_\ast(S^2_l)\in H_2(M_0)$ is non-zero and maps to $(1,b)$ and $(a,1)$, under $i$, for $l=1$ and $l=2$ respectively. Thus, the middle map $i$ is non-zero. This forces $H_3(M)$ to be free of rank $0$ or $1$. By Poincar\'{e} duality, the free part of $H_2(M)$ is of the same rank as $H_3(M)$. With the assumption that $H_2(M)$ has no torsion, the classification of simply-connected $5$-manifolds \cite{Sma61} imply that $S^5$ is the only such $5$-manifold with both $H_2(M)$ and $H_3(M)$ being zero. When both $H_2(M)$ and $H_3(M)$ are of rank $1$, then there are two manifolds - $S^3\times S^2$ and $X_\infty$ \cite{Bar65}. We cover all the three examples.\\
        \hspace*{0.5cm}(a) $f:S^5\to \R$, $f(x_0,\ldots,x_5):=x_0^2+x_1^2+x_2^2$ is Morse-Bott (cf. \autoref{eg:morseBottFunctionOnSphere}). In this case, $M_0=S^2\times S^2$.\\
        \hspace*{0.5cm}(b) $f:S^3\times S^2\to \R$, $f=h\,{\scriptstyle\circ}\, \textup{proj}_{S^3}$ is Morse-Bott, where $h:S^3\to\R$ is the height function.\\
        Even in this case, $M_0=S^2\times S^2$.  \\
        \hspace*{0.5cm}(c) We may consider the sphere bundle $\pi:S(\gamma\oplus\epsilon^1_\mathbb{C})\to S^2$. As it is classified by the generator of $\pi_1(SO(4))$, this sphere bundle is the only non-trivial (oriented) $S^3$-bundle over $S^2$. As this bundle has a section, both $M:=S(\gamma\oplus\epsilon^1_\mathbb{C})$ and $S^2\times S^3$ have the same cohomology ring. According to Barden's classification \cite[Lemma 1.1(v)]{Bar65}, where he denotes $M$ by $X_\infty$, $w_2(X_\infty)\neq 0$ while $w_2(S^2\times S^3)=0$. Thus, $M$ being non-spin and $S^2\times S^3$ being spin cannot be homotopy equivalent. The relevant Morse-Bott function on $M$ is defined as the fibrewise height function. Note that $M$ may be identified as the double fibrewise suspension of the Hopf fibration $S^3\to S^2$. It can be verified that $M_0$ in this case is $\mathbb{CP}^2\#\overline{\mathbb{CP}}^2$. 
    \end{eg}
    \begin{eg}\label{exotic}
    In the famous paper \cite{Mil56}, Milnor had constructed exotic $7$-spheres $M^7$ that arise as smooth (oriented) $S^3$-bundles over $S^4$, i.e., $M^7$ is homeomorphic to $S^7$ but not diffeomorphic to it. We may consider the Morse-Bott function $h\,{\scriptstyle\circ}\, \pi:M\to \R$, where $h:S^4\to\R$ is the height function. Note that $M_0$, being an $S^3$-bundle over $S^3$, is classified by $\pi_2(SO(4))=0$. This implies that $M_0$ is $S^3\times S^3$. 
    \end{eg}
    

\section{The case of projective spaces}

\begin{theorem}\label{thm:forComplexProjective}
    Let $M$ be a closed, smooth manifold of dimension $d$. Let $f$ be a Morse-Bott function on $M$ with only critical submanifolds $\mathbb{CP} ^k$ and $\mathbb{CP} ^l$ with $k<l$. Then $d = 2k + 2l + 2$ and $M$ is homotopic to $\mathbb{CP} ^{k+l+1}$.
\end{theorem}

\noindent The condition $k\neq l$ is necessary as the following example illustrates. Let $M = \mathbb{CP} ^2 \times S^k$ and $f:M \to  \mathbb{R} $ be defined as 
\[
    f([a], (x_1,\ldots,x_{k+1})) = x_{k+1}.
\]
Then $f$ is a Morse-Bott function with critical submanifolds $\mathbb{CP} ^2\times\{\pm e_{k+1}\}$. However, $M$ is not homotopy equivalent to $\mathbb{CP}^5 $ for any choice of $k$.

\vspace{0.3cm}
\hspace{0.3cm} From \autoref{eg:morseBottFunctionOnSphere}, we can construct an example of a Morse-Bott function on $\mathbb{CP}^{k+l+1} $ with critical submanifold $\mathbb{CP}^k $ and $\mathbb{CP}^l $.   

\begin{eg}\label{eg:morseBottFunctionOnCP}
    Define a map 
    \[
        f : \mathbb{C} ^{k+l+2} \to  \mathbb{R} , \ \mathbf{z} = (z_0, \dots, z_{k+l+2}) \mapsto  - \sum_{i=0}^{k} \vert z_i \vert ^2 + \sum_{i=k+1}^{k+l+2} \vert z_i \vert ^2
    \]
    and consider its restriction on $S^{2k+2l+3}$. Then, for any $\theta \in \mathbb{R} $, $f\left( e^{\iota \theta } \mathbf{z}  \right) = f(\mathbf{z} )$ and hence $f$ is $S^1$-invariant. Therefore, it induces a well-defined map $\tilde{f} : \mathbb{CP}^{k+l+1} \to  \mathbb{R} $. For any $[\mathbf{z} ] = [z_0: z_1 : \dots : z_{k+l+1}]$  
    \[
        \tilde{f} [\mathbf{z} ] = {\displaystyle  - \sum_{i=0}^{k} \vert z_i \vert ^2 + \sum_{i=k+1}^{k+l+2} \vert z_i \vert ^2}.
    \] 
    Then $\tilde{f} $ is a Morse-Bott function with critical submanifolds $\mathbb{CP} ^k = \tilde{f} ^{-1} (-1)$ and $\mathbb{CP} ^l = \tilde{f} ^{-1} (1)$. \\
    \hspace*{0.5cm} The same construction works for real numbers as well as quaternions. For example, consider the function from \autoref{eg:morseBottFunctionOnSphere}. It induces a Morse-Bott function on $\mathbb{RP}^{k+l+1}$ with critical submanifolds $\mathbb{RP} ^k$ and $\mathbb{RP} ^l$.
\end{eg}

\begin{proof}[Proof of \autoref{thm:forComplexProjective}]
    Each connected component of $M$ will either be a critical submanifold of $f$ or have at least two connected critical submanifolds of $f$ inside it. As we only have two connected critical submanifolds of $f$, this forces $M$ to be connected. \\
    \hspace*{0.5cm}Equip $M$ with a Riemannian metric. Let $\nu_{\mathbb{CP} ^k}, \nu_{\mathbb{CP}^l}$ denote the normal bundles of $\mathbb{CP} ^k$ and $\mathbb{CP}^l$ in $M$ respectively. Without loss of generality, let us assume that 
    \[
        f^{-1} (-1) = \mathbb{CP} ^k,~f^{-1} (1) = \mathbb{CP} ^l \text{ and } f^{-1} (0) \eqqcolon M_0.
    \] 
    Let $D(\xi)$ and $S(\xi)$ denote the unit disk bundle and the sphere bundle of $\xi$. By basic results from Morse-Bott theory, we have diffeomorphisms between $M_0, S(\nu_{\mathbb{CP}^k})$ and $S(\nu_{\mathbb{CP}^l})$. Moreover, 
    As $f^{-1}[-1,\epsilon)$ and $f^{-1}(-\epsilon,1]$ deformation retract to $\mathbb{CP}^k$ and $\mathbb{CP}^l$ respectively, by Seifert-van Kampen Theorem, we infer that $M$ is simply connected. Look at the fibrations $\mathcal{F}_k$ and $\mathcal{F} _l$ over $\mathbb{CP} ^k$ and $\mathbb{CP} ^l$ respectively. 
    \[
        \left.
            \begin{tikzcd}
                S ^{d-2k-1} \arrow[r,hook] & M_0\cong S(\nu_{\mathbb{CP}^k}) \arrow[d,"\pi"]\\
                                                & \mathbb{CP} ^k
            \end{tikzcd}
            \kern 0.2cm 
        \right\} \mathcal{F}_k
        \kern 1cm
        \left.
            \begin{tikzcd}
                S ^{d-2l-1} \arrow[r,hook] & M_0\cong S(\nu_{\mathbb{CP}^l}) \arrow[d]\\
                                               & \mathbb{CP} ^l
            \end{tikzcd}
            \kern 0.2cm 
        \right\} \mathcal{F}_l
    \]
    Let us outline the proof in the following steps. 
    \begin{enumerate}[(1)]
        \item[] \textbf{Step 1:} $M$ is of dimension $2(k + l +1)$ (\autoref{step-1CPn}).
        \item[] \textbf{Step 2:} The integral cohomology ring $H^\bullet(M;\mathbb{Z})$ is isomorphic to $\mathbb{Z} [\alpha ]/(\alpha ^{k+l+2})$, where $\vert \alpha \vert = 2 $ (\autoref{step-2CPn}).   
    \end{enumerate}
    Assuming steps 1 and 2, we want to show the existence of a homotopy equivalence $g: M \to \mathbb{CP}^{d/2}$. As $M$ is simply connected and a CW-complex of dimension $d$, by Whitehead's theorem, it is enough to show that there exists $g: M \to  \mathbb{CP}^{d/2} $ inducing a ring isomorphism in cohomology. Due to homotopy-theoretic definition of cohomology groups, 
    \begin{equation}
        H^2(M;\mathbb{Z} ) = [M, \mathbb{CP}^{\infty}  ],
    \end{equation}
    where $\mathbb{CP}^\infty$ is a model for the the Eilenberg-MacLane space $K(\mathbb{Z},2)$. Choose an $f: M \to  \mathbb{CP}^{\infty}  $ such that $\alpha = f^{\ast} (a)$ generates $H^2(M;\mathbb{Z})$, where $a$ is the generator of the cohomology ring $H^\ast(\mathbb{CP}^{\infty};\mathbb{Z})$. Since $M$ is a CW-complex, by the cellular approximation theorem, there exists a cellular map $g: M \to  \mathbb{CP}^{\infty}  $ such that $g \sim f$. Cellularity of $g$ implies that $g$ factors through $\mathbb{CP}^{d/2} $, that is, the following diagram is commutative
    \[
        \begin{tikzcd}
            M \arrow[r, "g"] \arrow[dr, "\tilde{g}"'] & \mathbb{CP}^{\infty}\\
                                                    & \mathbb{CP}^{d/2}\arrow[u, hook, "i"']
        \end{tikzcd}
    \]    
    Now we claim that $\tilde{g} : M \to \mathbb{CP}^{d/2} $ induces ring isomorphism in integral cohomology. By \hyperref[step-2CPn]{Step 2}, we know that  $H^\bullet (M) \cong \mathbb{Z} [\alpha ]/(\alpha ^{k+l+2})$. Also, $H^\bullet ( \mathbb{CP}^{d/2}  ) \cong \mathbb{Z} [b]/(b^{\frac{d}{2}+1})$, where $b = i^{\ast} (a)$. But by definition of $f$, $f^{\ast} (a) = \alpha $. As $g \sim f$, $f^{\ast} (a) = g^{\ast} (a)$. Also, 
    \[
        g = i \,{\scriptstyle\circ}\, \tilde{g} \implies \tilde{g}^{\ast} \left( i^{\ast} (a) \right) = f^{\ast} (a) = \alpha  \implies \tilde{g}^{\ast} (b) = \alpha .    
    \]       
    This completes the proof the theorem.
\end{proof}

\begin{lemma}\label{step-1CPn}
    With the conditions as in \autoref{thm:forComplexProjective}, the dimension of $M$ is $2k+2l+2$.
\end{lemma}
    
\begin{proof}

    At first, we claim that $d$ is even. If not, then $d-2k-1$ is even and we consider the second page of the Serre spectral sequence of $\mathcal{F} _k$, i.e., $E^2_{p,q}\cong H_p(\mathbb{CP}^k;\mathbb{Z})\otimes H_q(S^{d-2k-1};\mathbb{Z})$.
    \noindent It follows that there are no differentials in the $E^2$-page. In fact, there are no differentials in any page as $d-2k-1$ is even. Thus, 
    \[
        E^\infty _{p,q} \cong E^2_{p,q} \cong H_{p}(\mathbb{CP} ^k) \otimes H_q ( S^{d-2k - 1} ).
    \]   
    Similarly for the fibration $\mathcal{F} _l$, there are no differentials and we obtain 
    \[
        H_\bullet (\mathbb{CP}^k) \otimes H_\bullet (S ^{d-2k-1}) \cong  H_\bullet (M_0) \cong H_\bullet(\mathbb{CP}^l) \otimes H_\bullet (S ^{d-2l-1})
    \]
    as graded abelian groups. Now, consider the following homology,
    \begin{align*}
        H_{d-2l-1}(M_0) & \cong \big( H_{d-2l-1}(\mathbb{CP}^l ) \otimes _{\mathbb{Z} } H_0(S ^{d-2l-1}) \big) \oplus \big(H_0(\mathbb{CP}^l ) \otimes _{\mathbb{Z} } H_{d-2l-1}(S ^{d-2l-1}) \big)  \\
        & \cong H_{d-2l-1}(\mathbb{CP}^l ) \oplus \mathbb{Z} .
    \end{align*}
    We also have, 
    \begin{align*}
        H_{d-2l-1}(\mathbb{CP}^k) \cong  H_{d-2l-1}(M_0) \cong H_{d-2l-1}(\mathbb{CP}^l ) \oplus \mathbb{Z}.
    \end{align*}
    From the above relation, we have $d-2l-1 > 2l$ and $d-2l -1 \leq 2k <2l$, a contradiction. Thus, $d$ must be even. \\
    \hspace*{0.5cm} We now claim that $d-2k >2k$. If not, then we have
    \[
        d - 2l < d - 2k \leq 2k < 2l.
    \] 
    Look at the long exact sequence in the homotopy groups corresponding to $\mathcal{F} _l$:  
    \begin{align*}
        \mathcal{F} _l: & \kern 0.2cm \cdots \to \pi _{d-2l}(\mathbb{CP} ^l) \to  \pi_{d-2l-1}(S ^{d-2l-1}) \to \pi_{d-2l-1}(M_0) \to \pi_{d-2l-1}(\mathbb{CP}^l ) \to \cdots
    \end{align*} 
    In the above sequence, $\pi _{d-2l}(\mathbb{CP} ^l) \cong \pi_{d-2l}(S ^{2l+1}) \cong 0$ and $\pi _{d-2l-1}(\mathbb{CP} ^l) \cong \pi_{d-2l-1}(S ^{2l+1}) \cong 0$. Thus, 
    \[
        \pi_{d-2l-1}(M_0) \cong \pi_{d-2l-1}(S ^{d-2l-1}) \cong \mathbb{Z} .
    \]  
    We also have, from $\mathcal{F}_k$, the following: 
    \begin{center}
        \begin{tikzcd}[column sep=2em]
            \cdots \arrow[r] &  \cancelto{0}{\pi_{d-2l-1}(S^{d-2k-1})} \arrow[r] & \pi_{d-2l-1}(M_0) \arrow[r] & \cancelto{0}{\pi_{d-2l-1}(\mathbb{CP}^k)} \arrow[r] & \cdots
        \end{tikzcd}
    \end{center}
    where all the groups vanish because $d-2l < d- 2k\leq 2k$ and $\pi_{d-2l-1}(\mathbb{CP}^k)\cong\pi_{d-2l-1}(S^{2k+1})$. This implies that $\pi_{d-2l-1}(M_0)=0$, a contradiction. Hence, $d -2k > 2k$. \\
    \hspace*{0.5cm} Since $d-2k > 2k$, the fibration $\mathcal{F} _k$ has a section and as $d - 2k$ is even, $d- 2k \geq  2k + 2$. So, for any $j\geq 1$, 
    \begin{align*}
        \pi _j\left( M_0 \right) \cong \pi _j (\mathbb{CP}^k) \oplus \pi _j(S ^{d-2k-1}).
    \end{align*}    
    In particular, if $j= 2k + 1$, then
    \begin{align*}
        \pi _{2k+1}\left( M_0 \right) & \cong \pi _{2k+1} ( \mathbb{CP} ^{k}) \oplus \pi _{2k+1} (S ^{d-2k-1})\\
        & \cong
         \begin{cases}
            \mathbb{Z} \oplus \mathbb{Z} , &\text{ if } 2k + 1 = d-2k - 1 ;\\
            \mathbb{Z} , &\text{ otherwise}.
        \end{cases}
    \end{align*} 
    We now consider the long exact sequence of homotopy groups for $\mathcal{F} _l$. 
    \begin{center}
        \begin{tikzcd}[column sep= 2em]
            \cdots \arrow[r] & \pi _{2k+2}(\mathbb{CP}^l) \arrow[r] & \pi _{2k+1}(S ^{d-2l-1}) \arrow[r] & \pi _{2k+1}(M_0) \arrow[r] & \cancelto{0}{\pi _{2k+1}(\mathbb{CP}^l)} \arrow[r] & \cdots . 
        \end{tikzcd}
    \end{center}
    This implies that the map $\pi _{2k+1}(S ^{d-2l-1}) \to \pi _{2k+1}(M_0)$ must be surjective. Since $\pi _{2k+1}(M_0)$ is either $\mathbb{Z}$ or $\mathbb{Z} \oplus \mathbb{Z} $, we must have $\pi _{2k+1}(S ^{d-2l-1}) \cong \mathbb{Z}\cong\pi_{2k+1}(M_0) $. Therefore, $2k + 1 =d-2l-1$, which implies $d = 2k + 2l + 2$. 
    \end{proof} 
    \begin{lemma}\label{step-2CPn}
        With conditions as in \autoref{thm:forComplexProjective}, the ring $H^\bullet(M;\mathbb{Z})$ is $\mathbb{Z}[\alpha]/(\alpha^{k+l+2})$, where $|\alpha|=2$.  
    \end{lemma}
    \begin{proof}
    If $k=0$, then $M_0$ is a sphere and of dimension $2l+1$ by \autoref{step-1CPn}. Note that $M_0\cong S^{2l+1}$ is the total space of fibration over $\mathbb{CP}^l$, i.e., the Hopf fibration. This matches with the construction of $\mathbb{CP}^{l+1}$, obtained by gluing a $(2l+2)$-cell to the total space of the Hopf fibration. Thus, in this case, $M$ is homeomorphic to $\mathbb{CP}^{l+1}$.\\
    \hspace*{0.5cm}We may assume that $k\geq 1$. As $d-2k-1 = 2l + 1$, we look at the second page of the (cohomology) spectral sequence for the fibration $\mathcal{F} _k$. We conclude that 
    \[
        E_2^{p, q} = 
        \begin{cases}
            \mathbb{Z}, & 0 \leq p \leq 2k,\, p \text{ is even and } q \in \{ 0, 2l+1 \} \\ 
            0, & \text{ otherwise}.
        \end{cases}
    \] 
    While seeking potential non-zero differentials, we need $r$ so that $2p + r$ is an even number less than $2k$ and $2l + 1 - r + 1 = 0$ for $p = 0,2,\dots, 2k$. This implies $r = 2l+2$ and $2p + r = 2p + 2l + 2 \in \{ 0,2,\dots, 2k \} $. Therefore, 
    \begin{align*}
        & 0 \leq p + l + 1 \leq k \text{ for } p = 0,2,\dots, 2k \implies l + 1 \leq k,
    \end{align*}   
    which is not possible. Hence, there is no differential in any page and we have an isomorphism of graded abelian groups
    \begin{equation}\label{eq:cohomologyGroupOfM0}
        H^\bullet(M_0) \cong H^\bullet (\mathbb{CP}^k ) \otimes H^\bullet(S ^{2l+1}).
    \end{equation}
    Since the fibration $\mathcal{F} _k$ has a section, say $s$, we have 
    \[
        \begin{tikzcd}
            H^\bullet (\mathbb{CP}^k) \arrow[r, hook, "\pi^{\ast}"] \arrow[rr, bend right, "\textup{id} "] & H^\bullet(M_0) \arrow[r, "s^{\ast}"] & H^\bullet(\mathbb{CP}^k).
        \end{tikzcd}        
    \]
    We know that the cohomology ring $H^\bullet (\mathbb{CP}^k ) \cong \mathbb{Z} [a ]/(a^{k+1} )$, where $\vert a  \vert =2 $. Thus, $\pi^{\ast} (a ^k) \in H^{2k} (M_0)$ is a generator. Since $M_0$ is simply connected, it is oriented. Poincar\'{e} duality implies that there exists $\beta \in H^{2l+1}(M_0)$ such that $\pi ^{\ast} (\alpha ^k) \smile \beta $ is in the top class of $H^\bullet(M_0)$. This implies $\pi ^{\ast} (\alpha ^j) \smile \beta \neq 0$ for $k\geq j\geq 0$. Since $\beta $ is a generator of $H^{2l+1}(S ^{2l+1})$, we have $\beta \smile \beta =0$. Thus, we have the following isomorphism of rings 
    \begin{equation}\label{eq:cohomologyRingOfM0}
        H^\bullet(M_0) \cong H^\bullet(\mathbb{CP}^k) \otimes H^\bullet(S ^{2l+1}).
    \end{equation}    
    \hspace*{0.5cm}We consider the open cover $U=f^{-1}[-1,1/2)$ and $V=f^{-1}(-1/2,1]$. Observe that   
    \[
        U \simeq \mathbb{CP}^k, \quad V \simeq \mathbb{CP}^l,\quad U  \cap V \simeq M_0.   
    \] 
    Then from the Mayer-Vietoris sequence applied to the cover $\{U,V\}$ of $M$, we have the following long exact sequence. 
    \begin{equation}\label{MV-for-M}
        \cdots \rightarrow H^j(M_0) \xrightarrow{\delta} H^{j+1}(M) \rightarrow H^{j+1}(\mathbb{CP}^k) \oplus H^{j+1}(\mathbb{CP}^{l}) \rightarrow H^{j+1}(M_0) \rightarrow H^{j+2}(M)\rightarrow \cdots .
    \end{equation}
    If $j\geq 2l$, then \eqref{MV-for-M} implies that the cohomology of $M$ in degree at least $2l+1$ is zero in odd degrees and $\Z$ in even degrees. \\
    \hspace*{0.5cm}Let us put $j= 1$ in \eqref{MV-for-M}. 
    

    \begin{center}
        \begin{tikzcd}
            0 \arrow[r, "\delta"] & H^2(M) \arrow[r, "\iota"] & \mathbb{Z} \oplus \mathbb{Z} \arrow[r, "j"] & \mathbb{Z} \arrow[r] & H^3(M) \arrow[r] & 0.
        \end{tikzcd}


            
            
    \end{center}
    The generator $a\in H^2(\mathbb{CP}^k)$ maps onto the generator of $H^2(M_0)$ due to \eqref{eq:cohomologyGroupOfM0} and $\mathcal{F}_k$ admitting a section. Thus, $H^3(M)=0$ and $H^2(M)\cong\mathbb{Z}$, generated by $\alpha$. Let $b$ generate the cohomology of $H^\bullet(\mathbb{CP}^l)$. This element $b$ arises from $\iota:\mathbb{CP}^1\hookrightarrow \mathbb{CP}^l$, i.e, $\iota^\ast(b)$ is the generator of $H^2(\mathbb{CP}^1;\mathbb{Z})$. Consider the commutative diagram 
    \begin{center}
        \begin{tikzcd}[column sep = 1cm]
            M_0\big|_{\mathbb{CP}^1} & M_0 &  \nu_{\mathbb{CP}^l} \\
            \mathbb{CP}^1  & \mathbb{CP} ^l & \mathbb{{CP}}^l 
            \arrow["", hook, from = 1-1, to = 1-2]
            \arrow["",  hook, from = 1-2, to = 1-3]
            \arrow["", hook, from = 2-1, to = 2-2]
            \arrow["\textup{id}", from = 2-2, to = 2-3]
            \arrow["", from = 1-1, to = 2-1]
            \arrow["\pi", from = 1-2, to = 2-2]
            \arrow["", from = 1-3, to = 2-3]
        \end{tikzcd}
    \end{center}
    The leftmost vertical arrow admits a section while the rightmost vertical arrow is a homotopy equivalence via the zero section. Thus, all induced maps in $H^2$ are isomorphisms and $b\in H^2(\mathbb{CP}^l)$ maps to the generator of $H^2(M_0)$. Thus, kernel of $j$ is given by $(a,b)$ and $\iota(\alpha)=(a,b)$. \\
    \hspace*{0.5cm}As $H^r(M_0)$ is that of $H^r(\mathbb{CP}^k)$ for $r\leq 2k$, from the Mayer-Vietoris sequence (with $2r\leq 2k$)
    \[
        0 \longrightarrow H^{2r}(M) \longrightarrow H^{2r}(\mathbb{CP}^k) \oplus H^{2r}(\mathbb{CP}^{l}) \longrightarrow H^{2r}(M_0) \longrightarrow 0
    \]
    we deduce that that that the cohomology of $M$ in degree at most $2k+1$ is zero in odd degrees and $\Z$ in even degrees. If $2k<j< 2l$, then $H^j(M_0), H^{j+1}(M_0)$ and $H^{j+1}(\mathbb{CP}^k)$ are zero. This forces $H^{j+1}(M)$ to be isomorphic to $H^{j+1}(\mathbb{CP}^l)$. In particular, $H^{2r}(M)=\mathbb{Z}\alpha^r$, if $r\leq l$. By Poincar\'e duality, there exists $\beta\in H^{2k+2}(M)$ such that $\alpha^l\cup \beta$ generates $H^d(M)$. This forces $\beta\in\{\pm \alpha^{k+1}\}$ and the claim about the ring structure of $H^\bullet(M)$ follows.
    \end{proof}

\hspace*{0.5cm}Arguments similar to those used in \autoref{thm:forComplexProjective} may be used to prove the quaternionic case.
\begin{theorem}\label{thm:forQuaternionicProjective}
    Let $M$ be a closed, smooth manifold of dimension $d$. Let $f$ be a Morse-Bott function on $M$ with only critical submanifolds $\mathbb{HP} ^k$ and $\mathbb{HP} ^l$ with $k<l$. Then $d = 4(k + l + 1)$ and $M$ is homotopic to $\mathbb{HP} ^{k+l+1}$.
\end{theorem}

\hspace*{0.5cm}The analogous statement for the real projective spaces is also true (see \autoref{thm:forRealProjective} below). However, the presence of non-zero fundamental groups force us to modify the arguments appropriately. We will consider homology and cohomology with coefficients in $\mathbb{Z} _2$. This enables us to apply Poincar\'{e} duality with $\mathbb{Z}_2$-coefficients. 

\begin{theorem}\label{thm:forRealProjective}
    Let $M$ be a closed, smooth manifold of dimension $d$. Let $f$ be a Morse-Bott function on $M$ with only critical submanifolds $\mathbb{RP} ^k$ and $\mathbb{RP} ^l$ with $k<l$. Then $d = k +l + 1$ and $M$ is homotopic to $\mathbb{RP} ^{k+l+1}$.
\end{theorem}
 \begin{proof}
 We may assume that 
 \[
    f^{-1} (-1) = \mathbb{RP} ^k,\ f^{-1} (1) = \mathbb{RP} ^l, \text{ and } f^{-1} (0) \eqqcolon M_0.
 \]
 Then, we have the following fibrations: 
 \[
        \left.
            \begin{tikzcd}
                S ^{d-k-1} \arrow[r,hook] & M_0 \arrow[d]\\
                                                & \mathbb{RP} ^k
            \end{tikzcd}
            \kern 0.2cm 
        \right\} \mathcal{F}_k
        \kern 1cm
        \left.
            \begin{tikzcd}
                S ^{d-l-1} \arrow[r,hook] & M_0 \arrow[d]\\
                                               & \mathbb{RP} ^l
            \end{tikzcd}
            \kern 0.2cm 
        \right\} \mathcal{F}_l
    \]
 The case $k=0$ can be handled exactly as in $k=0$ case in the proof of \autoref{step-2CPn}. In this case, $M$ will be homeomorphic to $\mathbb{RP}^{l+1}$. We shall assume that $k\geq 1$ and complete the proof, assuming the following steps. 
 \begin{enumerate}[(1)]
    \item[] {\bf Step 1:} $M$ is of dimension $k + l + 1$ and $\pi_1(M)=\mathbb{Z}_2$ (\autoref{step-1RPn}). 
    \item[] {\bf Step 2:} The universal cover of $M$ is $S^{k+l+1}$ and the mod $2$ cohomology ring of $M$ is isomorphic to $\mathbb{Z} _2 [\alpha ]/(\alpha ^{k+l+2})$, where $\vert \alpha \vert = 1$ (\autoref{step-2RPn}).  
 \end{enumerate}
 Due to Step 2, there exists a fixed point free involution $T:S^{k+l+1}\to S^{k+l+1}$ such that $M$ arises as the orbit space of $S^{k+l+1}$ under this map. It is well-known (see Lemma 3 of \cite{HirMil64}, for instance) that such an orbit space is homotopy equivalent to $\mathbb{RP}^{k+l+1}$.  
    \end{proof}

\begin{lemma}\label{step-1RPn}
    With the conditions as in \autoref{thm:forRealProjective} and $k\geq 1$, the dimension of $M$ is $k+l+1$ and $\pi_1(M)=\mathbb{Z}_2$.
\end{lemma}
    
\begin{proof}
If $d=l+1$, then $\mathcal{F}_k$ implies that $M_0$ is connected. Using $\mathcal{F}_l$, $M_0$ is a connected double cover of $\mathbb{RP}^l$, whence it must be $S^l$. The long exact sequence in homotopy groups associated with $\mathcal{F}_k$ now implies that $\pi_1(\mathbb{RP}^k)=0$. This contradiction establishes that $d>l+1$. 
If we assume that $d-k\leq k$, then $$1< d-l<d-k\leq k<l.$$ This forces $M_0$ to be connected and $\pi_{d-l}(\mathbb{RP}^l)\cong\pi_{d-l}(S^l)=0$. From the long exact sequence associated to $\mathcal{F}_l$ we obtain $$\mathbb{Z}\cong\pi_{d-l-1}(S^{d-l-1})\leq \pi_{d-l-1}(M_0).$$ From the long exact sequence associated to $\mathcal{F}_k$, we obtain 
\[\pi_{d-l-1}(M_0)\leq \pi_{d-l-1}(\mathbb{RP}^k)\cong\left\{\begin{array}{rl}
\mathbb{Z}_2 & \textup{if $d-l-1=1$}\\
\pi_{d-l-1}(S^k)=0 & \textup{if $d-l-1>1$}.
\end{array}\right.
\]
This contradiction establishes that $d-k>k$. This implies that $\mathcal{F}_k$ admits a section and that $M_0$ is connected. Thus, 
\begin{equation}
    \pi_j(M_0)\cong\pi_j(\mathbb{RP}^k)\oplus\pi_j(S^{d-k-1})\label{M_0RPk}
\end{equation}
and $\pi_k(M_0)$ is free abelian of rank $2$ or $1$ if $d-k-1=k$ or $d-k-1>k$ respectively. The long exact sequence associated to $\mathcal{F}_l$ gives 
\begin{equation}
    \cdots\longrightarrow \pi_{j+1}(\mathbb{RP}^l)\longrightarrow \pi_j(S^{d-l-1})\longrightarrow \pi_j(M_0)\longrightarrow \pi_j(\mathbb{RP}^l)\longrightarrow \cdots \label{M_0RPl}
\end{equation}
With $1\leq k<l$, $\pi_k(\mathbb{RP}^l)$ is either zero or $\mathbb{Z}_2$, and $\pi_{k+1}(\mathbb{RP}^l)$ is either zero or $\mathbb{Z}$. As $\pi_k(S^{d-l-1})$ contains at most one $\mathbb{Z}$, we must have $d-k-1>k$, $\pi_k(M_0)\cong\mathbb{Z}$ and $\mathbb{Z}\leq \pi_k(S^{d-l-1})$. This forces one of two cases:\\
\hspace*{0.5cm}(i) $k=d-l-1$;\\
\hspace*{0.5cm}(ii) $d-l-1$ is even and $k=2(d-l-1)-1$.\\
In case (ii), $l>d-l$ and we write out \eqref{M_0RPl} for $j=d-l-1$:
\begin{center}
        \begin{tikzcd}[column sep= 2em]
            \cdots \arrow[r] & \cancelto{0}{\pi_{d-l}(\mathbb{RP}^l)} \arrow[r] & \pi_{d-l-1}(S^{d-l-1}) \arrow[r] & \pi_{d-l-1}(M_0) \arrow[r] & \cancelto{0}{\pi_{d-l-1}(\mathbb{RP}^l)} \arrow[r] & \cdots . 
        \end{tikzcd}
    \end{center}
However, by \eqref{M_0RPk} for $j=d-l-1$, we get
$$\pi_{d-l-1}(M_0)\cong \pi_{d-l-1}(\mathbb{RP}^k)\cong\pi_{d-l-1}(S^k)=0.$$
This contradiction implies that case (i), i.e., $d=k+l+1$ as the only possibility.\\
\hspace*{0.5cm}We consider the open cover $U=f^{-1}[-1,1/2)$ and $V=f^{-1}(-1/2,1]$ of $M$. Te inclusion $i^U:M_0\hookrightarrow  U$ induces an isomorphism 
\begin{equation*}\label{pi1_RPk}
    i^U_\ast:\pi_1(M_0)\to \pi_1(U)\cong\pi_1(\mathbb{RP}^k)
\end{equation*}
due to \eqref{M_0RPk} and $U$ deforming to $\mathbb{RP}^k$. It follows from Seifert-van Kampen Theorem that the inclusion $\iota:V\hookrightarrow M$ induces an isomorphism
\begin{equation}\label{pi1RPl}
    \iota_\ast:\pi_1(V)\stackrel{\cong}{\longrightarrow}\pi_1(M).
\end{equation} 
As $V$ deforms to $\mathbb{RP}^l$, the claim follows.
\end{proof}
\begin{lemma}\label{step-2RPn}
    With the conditions as in \autoref{thm:forRealProjective} and $k\geq 1$, The universal cover of $M$ is $S^{k+l+1}$ and the mod $2$ cohomology ring of $M$ is isomorphic to $\mathbb{Z} _2 [\alpha ]/(\alpha ^{d+1})$, where $\vert \alpha \vert = 1$
\end{lemma}
\begin{proof} 
Consider the commutative diagram 
    \begin{center}
        \begin{tikzcd}[column sep = large, row sep = large]
            M_0\big|_{\mathbb{RP}^1} & M_0 & \nu_{\mathbb{RP}^l} \\
            \mathbb{RP}^1 & \mathbb{RP}^l & \mathbb{RP}^l
            \arrow["", hook, from = 1-1, to = 1-2]
            \arrow["", hook, from = 1-2, to = 1-3]
            \arrow["i", hook, from = 2-1, to = 2-2]
            \arrow["\textup{id}", from = 2-2, to = 2-3]
            \arrow["\pi"', from = 1-1, to = 2-1]
            \arrow["\pi"', from = 1-2, to = 2-2]
            \arrow[""', from = 1-2, to = 2-2]
            \arrow["\simeq", from = 1-3, to = 2-3]
        \end{tikzcd}
    \end{center}
    Note that the restriction of $M_0$ to $\mathbb{RP}^1$ is a fibre bundle with fibre $S^k$. If $k>1$, then there is a section and $\pi$ induces an isomorphism on fundamental groups. If $k=1$, then the total space is the torus or the Klein bottle and in both cases, $\pi$ induces a surjection on fundamental groups. As $\iota:\mathbb{RP}^1\hookrightarrow \mathbb{RP}^l$ induces a surjection on $\pi_1$, and $\nu_{\mathbb{RP}^l}\subseteq V$ is a homotopy equivalence, we have the following induced diagram
    \begin{center}
        \begin{tikzcd}[column sep = large, row sep = large]
            \pi_1\big(M_0\big|_{\mathbb{RP}^1}\big) & \pi_1(M_0) & \pi_1(V) \\
             \mathbb{Z} & \mathbb{Z}_2 & \mathbb{Z}_2
            \arrow["", from = 1-1, to = 1-2]
            \arrow["i_\ast", ->>, from = 2-1, to = 2-2]
            \arrow["\pi_\ast"', ->>, from = 1-1, to = 2-1]
            \arrow["i^V_\ast", from = 1-2, to = 1-3]
            \arrow["", from = 1-2, to = 2-2]
            \arrow["\textup{id}", from = 2-2, to = 2-3]
            \arrow["\cong", from = 1-3, to = 2-3]
        \end{tikzcd}
    \end{center}
This forces the middle vertical map as well as $i^V_\ast$ to be surjective maps.\\
\hspace*{0.5cm}We consider the universal cover $p:\widetilde{M}\to M$. The generator of $\pi_1(\mathbb{RP}^l)$, being the same as that of $\pi_1(M)$ due to \eqref{pi1RPl}, is lifted to a path in $\widetilde{M}$. This implies that $p^{-1}(\mathbb{RP}^l)=S^l$. As the generator of $\pi_1(\mathbb{RP}^k)$ maps to the generator of $\pi_1(\mathbb{RP}^l)$ (due to $i_\ast$ being surjective in the above diagram), by homotopy lifting, it also lifts to a path. Thus, $p^{-1}(\mathbb{RP}^k)\cong S^k$. Therefore, $\widetilde{M}$ is a closed, connected manifold equipped with a Morse-Bott function $f\,{\scriptstyle\circ}\, p$ with two critical submanifolds $S^k$ and $S^l$. \autoref{thm:forSpheres} implies that $\widetilde{M}$ is homeomorphic to $S^{k+l+1}$. The Gysin sequence in cohomology with mod $2$ coefficients now imply that $H^\bullet(M;\mathbb{Z}_2)\cong \mathbb{Z}_2[\alpha]/(\alpha^{d+1})$, where $|\alpha|=1$.
\end{proof}

\begin{remark}
    We have seen that (refer \eqref{eq:cohomologyRingOfM0}) that $M_0$ has the same cohomology ring as that of $\mathbb{CP}^k \times S^{2l+1}$. In \autoref{eg:morseBottFunctionOnCP}, $\nu_{\mathbb{CP}^k}=(l+1)\gamma^\ast$, where $\gamma$ is the tautological (complex) line bundle over $\mathbb{CP}^k$. Consider the projection map $p:S^{2k+1}\times \mathbb{C}\to S^{2k+1}$. Equip the the codomain with the standard $S^1$-action and the domain with the diagonal $S^1$-action. The induced map 
    $$\tilde{p}:S^{2k+1}\!\times_{S^1}\! \mathbb{C}\to \mathbb{CP}^k$$
    defines a complex line bundle. The map
    $$\varphi:S^{2k+1}\!\times_{S^1}\! \mathbb{C}\to \gamma^\ast,\,\,[(v,\lambda)]\mapsto (zv\mapsto z\lambda)$$
    defines an isomorphism of complex line bundles over $\mathbb{CP}^k$. This implies an isomorphism of bundles 
    $$\varphi:S^{2k+1}\!\times_{S^1}\! \mathbb{C}^r\stackrel{\cong}{\longrightarrow} r\gamma^\ast$$
    and consequently, $M_0\cong S(\nu_{\mathbb{CP}^k})\cong S^{2k+1}\!\times_{S^1}\! S^{2l+1}$. It is natural to ask whether $M_0$, in the general case, is always homeomorphic (or even homotopic) to $S^{2k+1}\!\times_{S^1}\! S^{2l+1}$.\\
    \hspace*{0.5cm} For the real case, in the last part of \autoref{eg:morseBottFunctionOnCP}, for the standard Morse-Bott function on $\mathbb{RP}^{k+l+1}$, we infer that $M_0$ is $S^{k}\!\times_{\mathbb{Z}_2}\! S^{l}$. It is natural to ask whether $M_0$, in the general case, is always homeomorphic (or even homotopic) to $S^{k}\!\times_{\mathbb{Z}_2}\! S^{l}$.
\end{remark}

\section*{Acknowledgement} 
We are greatly inspired and influenced, like so many before us, by the seminal paper \cite{Mil56} of John Milnor on exotic $7$-spheres. The title of this article is our humble tribute to this great scholar and his beautiful mathematical creations. Part of this work was done while the second named author was a joint postdoctoral fellow at Jilin University and G\"{o}ttingen University. The author acknowledges the funding support provided by Jilin University.

\bibliographystyle{alphaurl}

\end{document}